\documentclass[a4paper,twoside,12pt]{amsart}

\usepackage[english]{babel}
\usepackage[latin1]{inputenc}

\usepackage{amsmath,amsthm,amssymb}
\numberwithin{equation}{section}
\usepackage{amsfonts}
\usepackage{multicol}
\usepackage{enumerate}
\usepackage{graphicx}
\usepackage{array}
\usepackage{color}
\usepackage{comment}

\def\Z{\mathbb Z}

\def\R{\mathbb R}

\def\a{\alpha}
\def\b{\beta}
\def\s{\sigma}
\def\l{\lambda}

\def\g{\gamma}

\def\CF{\mathcal{F}}

\def\CE{\mathcal{E}}
\def\cM{M}
\def\la{\langle}
\def\ra{\rangle}

\newtheorem{teor}{Theorem}[section]
\newtheorem{theorem}[teor]{Theorem}

\newtheorem{corollary}[teor]{Corollary}

\newtheorem{proposition}[teor]{Proposition}
\newtheorem{lemma}[teor]{Lemma}

\title[Highly oscillatory unimodular Fourier multipliers] {Highly oscillatory unimodular Fourier multipliers on modulation spaces}
\author{Fabio Nicola, Eva Primo, Anita Tabacco}
\address{Dipartimento di Matematica,
Politecnico di Torino, corso Duca degli
Abruzzi 24, 10129 Torino, Italy}
\address{Departament d'An\`alisi Matem\`atica, Universitat de Val\`encia, Dr. Moliner 50, 46100-Burjassot, Val\`encia (Spain)}
\address{Dipartimento di Matematica,
Politecnico di Torino, corso Duca degli
Abruzzi 24, 10129 Torino, Italy}
\email{fabio.nicola@polito.it}
\email{eva.primo@uv.es}
\email{anita.tabacco@polito.it}
\thanks{}
\subjclass[2010]{42B15 (35Q55 42B35)}
\keywords{Fourier multipliers, modulation spaces, Wiener amalgam spaces}

\begin{document}
\begin{abstract} We study the continuity on the modulation spaces $M^{p,q}$ of Fourier multipliers with symbols of the type $e^{i\mu(\xi)}$, for some real-valued function $\mu(\xi)$. A number of results are known, assuming that the derivatives of order $\geq 2$ of the phase $\mu(\xi)$ are bounded or, more generally, that its second derivatives belong to the Sj\"ostrand class $M^{\infty,1}$. Here we extend those results, by assuming that the second derivatives lie in the bigger Wiener amalgam space $W(\CF L^1,L^\infty)$; in particular they could have stronger oscillations at infinity such as $\cos |\xi|^2$. Actually our main result deals with the more general case of possibly unbounded second derivatives. In that case we have boundedness on weighted modulation spaces with a sharp loss of derivatives.
\end{abstract}

\maketitle
\section{Introduction}
Fourier multipliers represent one of the main research field in Harmonic Analysis, where a number of challenging problems remain open \cite{stein93}. The connections with other branches of pure and applied mathematics are uncountable (combinatorics, PDEs, signal processing, functional calculus, etc.). In this paper we consider Fourier multipliers in $\R^d$ of the form 
\begin{equation}\label{moltip}
e^{i\mu(D)}f(x):=\int_{\mathbb{R}^d}
e^{2\pi
ix\xi}e^{i\mu(\xi)}\hat{f}(\xi)\, d\xi
\end{equation}
for some real-valued phase $\mu$. \par
The prototype is given by the phase $\mu(\xi)=|\xi|^2$. In that case the operator $e^{i\mu(D)}$ is the propagator for the free Schr\"odinger equation, and similarly for other constant coefficient equations. Hence it is of great interest to study the continuity of such operators on several functions spaces arising in PDEs. Whereas such operators represent unitary transformations of $L^2(\R^d)$, their continuity on $L^p(\R^d)$ for $p\not=2$ in general fails. Hence recently a number of works addressed the problem of the continuity in other function spaces. Among those, the more convenient spaces, at least in the case of the Schr\"odinger model, turned out to be the modulation spaces $M^{p,q}(\R^d)$, $1\leq p,q\leq\infty$, widely used in Time-frequency Analysis  \cite{F1,grochenig}. The basic reason is that the Schr\"odinger  propagator is sparse with respect to Gabor frames \cite{CNR2009}, which in turn give a discrete characterization of the modulation space norms. Now, for $p=q=2$ we have $M^{2,2}(\R^d)=L^2(\R^d)$ and, in general, the modulation space norm gives a measure of the size of a function or temperate distribution {\it in phase space} or {\it time-frequency plane}, exactly as the Lebesgue space norms $L^p$ provide a measure of the size of a function in the physical space. The couple of indices $p,q$ allows one to measure the decay both in the physical and frequency domain, separately. We refer to the next section for the definition and basic properties of modulation spaces. \par
Now, it was proved in \cite{benyi} that the Schr\"odinger propagator (hence $\mu(\xi)=|\xi|^2$ in \eqref{moltip}) is bounded $M^{p,q}(\R^d)\to M^{p,q}(\R^d)$, for every $1\leq p,q\leq\infty$. This result motivated the study of the continuity of more general unimodular Fourier multipliers on modulation spaces. The recent bibliography in this connection is quite large; see e.g.\ \cite{benyi,benyi3,bib9,bib8,Concetti_2009_Schatten,CN2009,bib5,bib7,bib0,bib6,MNRTT,bib3,bib1,bib2,bib4}.
In short, it turns out that, for unbounded (smooth enough) phases, the properties which play a key role are:
\begin{center}
 {\it Growth and  oscillations of the second derivatives $\partial^\gamma \mu$, $|\gamma|=2$.}
 \end{center}
To put our results in context, let us just recall three basic facts. \par

\begin{enumerate}
\item[(a)] {\it No growth, mild oscillations} \cite[Theorem 11]{benyi}. 
{\it Suppose that
\[
|\partial^\gamma \mu(\xi)|\leq C\quad {\it for}\ \xi\in\R^d,\ 2\leq |\gamma|\leq 2(\lfloor d/2 \rfloor+1).
\] 
Then $e^{i\mu(D)}: M^{p,q}(\R^d)\to M^{p,q}(\R^d)$ is bounded for every $1\leq p,q\leq \infty$.}
\end{enumerate}
This result generalizes the case of the Schr\"odinger propagator, where the second derivatives of $\mu$ are in fact constants. 
\begin{enumerate}
\item[(b)] {\it No growth, mild oscillations} \cite[Lemma 2.2]{Concetti_2009_Schatten}. {\it Suppose that
\[
\partial^\gamma \mu\in M^{\infty,1}(\R^d) \quad{\it for}\ |\gamma|=2.
\]
Then $e^{i\mu(D)}: M^{p,q}(\R^d)\to M^{p,q}(\R^d)$ is bounded for every $1\leq p,q\leq \infty$.}
\end{enumerate}
Actually, \cite[Lemma 2.2]{Concetti_2009_Schatten} provides a partial but key result in this connection, from which it is easy to deduce that the symbol $\sigma(\xi)=e^{i\mu(\xi)}$ is then in the Wiener amalgam space $W(\CF L^1,L^\infty)$ (see below for the definition), which is sufficient to conclude (see also \cite{boulk,concetti-toft,TCG}). The result in (b) is also a particular case of \cite[Theorem 2.3]{CNR2015} where Schr\"odinger equations with rough Hamiltonians were considered.\par
Observe that the result in (b) improves that in (a), because of the embedding $C^{d+1}(\R^d)\hookrightarrow M^{\infty,1}(\R^d)$ (\cite[Theorem 14.5.3]{grochenig}). We also notice that $M^{\infty,1}(\R^d)\subset L^\infty(\R^d)$, so that here the second derivatives of $\mu$ do not grow at infinity, but they could oscillate, say, as $\cos |\xi|^\alpha$, with $0<\alpha\leq 1$ (cf.\ \cite[Corollary 15]{benyi}). 
\begin{enumerate}
\item[(c)] {\it Growth at infinity, mild oscillations} \cite[Theorem 1.1]{MNRTT}.
{\it Let $\alpha\geq 2$, and suppose that
\[
|\partial^\gamma \mu(\xi)|\leq C\langle \xi\rangle^{\alpha-2} \quad{\it for}\ 2\leq |\gamma|\leq \lfloor d/2\rfloor +3.
\]
Then $e^{i\mu(D)}: M^{p,q}_\delta(\R^d)\to M^{p,q}(\R^d)$ is bounded for every $1\leq p,q\leq \infty$ and $\delta\geq d(\alpha-2)|1/p-1/2|$. }
\end{enumerate}
Here $M^{p,q}_\delta$ is a modulation space weighted in frequency, so that we have in fact a loss of derivatives, which is proved to be sharp. \par
Now, it was proved in \cite[Lemma 8]{benyi} that, more generally, the operator $e^{i\mu(D)}$ is bounded on all $M^{p,q}(\R^d)$ for every $1\leq p,q\leq \infty$ if its symbol $e^{i\mu(\xi)}$ belongs to the Wiener amalgam space $W(\CF L^1,L^\infty)(\R^d)$ \cite{F2}, whose norm is defined as 
\[
\|f\|_{W(\CF L^1,L^\infty)}=\sup_{x\in\R^d}\|g(\cdot-x)f\|_{\CF L^1}
\]
where $g\in \mathcal{S}(\R^d)\setminus\{0\}$ is an arbitrary window.  This suggests to look at conditions on $\mu(\xi)$ in terms of this space, rather than modulation spaces. Here is our first result in this direction.

\begin{theorem} (No growth, strong oscillations).\label{lem:A_bound}
Let
$\mu\in C^2(\R^d)$,
real-valued, satisfying 
$$\partial^{\g}\mu(\xi) \in W(\CF L^1, L^{\infty})(\R^d)\quad {\it for}\ |\g|=2.$$ Then
$$ e^{i \mu (D)}: M^{p,q}(\R^d) \to M^{p,q}(\R^d)$$ is bounded for every $1 \leq p, q \leq \infty$.
\end{theorem}
Observe that $M^{\infty,1}(\R^d)\subset W(\CF L^1, L^{\infty})(\R^d)\subset L^\infty(\R^d)$ so that this result improves that in (b) above. Here the second derivatives of $\mu$ are still bounded, but they are allowed to oscillate, say, as $\cos |\xi|^2$ (cf.\ \cite[Theorem 14]{benyi}). This result is strongly inspired by \cite[Lemma 2.2]{Concetti_2009_Schatten} and in fact the proof is similar. However, our main result deals with the case of possibly unbounded second derivatives, as stated in the following theorem.
\begin{theorem} (Growth at infinity, strong oscillations). \label{caso_pq}
Let $\alpha\geq2$. Let 
$\mu\in C^2(\R^d)$,
real-valued and such that 
\[
\langle \xi \rangle^{2-\a} \partial^{\g}\mu(\xi) \in W(\CF L^1, L^{\infty})(\R^d)\quad{\it for}\ |\g|=2.
\]
Then
\[
e^{i\mu(D)}:M^{p,q}_{\delta}(\R^d)\rightarrow
 M^{p,q}(\R^d)
 \]
 is bounded
 for every $1\leq
p,q\leq\infty$ and
  \begin{equation}
  \label{soglia}
  \delta \geq d(\alpha-2)\left|\frac{1}{p}-\frac{1}{2}
  \right|.
  \end{equation}
\end{theorem}
The above threshold for $\delta$ agrees with that in (c), and also with the examples in \cite[Theorem 16]{benyi}, where even stronger oscillations were considered, but only for model cases.   \par
Theorem \ref{lem:A_bound} is of course a particular case of Theorem  \ref{caso_pq} and will be used as a step in the proof of the latter. \par
There are a number of easy extensions that could be considered. For example, the conclusion of Theorem  \ref{lem:A_bound} still holds for a phase of the type $\mu(\xi)=\mu_1(\xi)+\mu_2(\xi)$, where $\mu_1$ satisfies the assumption in that theorem and $\mu_2\in W(\CF L^1, L^{\infty})(\R^d)$. Hence one could allow phases that in a compact set do not have any derivatives but only $\CF L^1$ regularity, such as a positively homogeneous function of order $r>0$. We leave these easy developments to the interested reader. \par\medskip
In short the paper is organized as follows. In Section \ref{sec2} we collected a number of definitions and auxiliary results. Section \ref{sec3} is devoted to the proof of Theorem \ref{lem:A_bound}, whereas in Section \ref{sec4} we prove Theorem \ref{caso_pq}. \par\medskip

\textbf{Notation.} We define
$|x|^2=x\cdot x$, for
$x\in\R^d$, where $x\cdot
y=xy$ is the scalar product
on $\R^d$. We set $B_R(x_0)$
for the open ball in $\R^d$
of centre $x_0$ and radius
$R$. The space of smooth
functions with compact
support is denoted by
$C_0^\infty(\R^d)$, the
Schwartz class is
$\mathcal{S}(\R^d)$, the space of
tempered distributions
$\mathcal{S}'(\R^d)$.    The Fourier
transform is normalized to be
${\hat
  {f}}(\xi)=\CF f(\xi)=\int
f(t)e^{-2\pi i t\xi}dt$.
 Translation and modulation operators are defined, respectively, by
$$ T_xf(t)=f(t-x)\quad{\rm and}\quad M_{\xi}f(t)= e^{2\pi i \xi
 t}f(t).$$
We have the formulas
$(T_xf)\hat{} = M_{-x}{\hat
{f}}$, $(M_{\xi}f)\hat{}
=T_{\xi}{\hat {f}}$, and
$M_{\xi}T_x=e^{2\pi i
x\xi}T_xM_{\xi}$. The inner
product of two functions
$f,g\in L^2(\R^d)$ is $\la f, g
\ra=\int_{\R^d}f(t)
\overline{g(t)}\,dt$, and its
extension to $\mathcal{S}'(\R^d)\times \mathcal{S}(\R^d)$
will be also denoted by  $\la
\cdot, \cdot \ra$. The
notation $A\lesssim B$ means
$A\leq c B$ for a suitable
constant $c>0$ depending only on the dimension $d$ and Lebesgue exponents $p,q,\ldots,$ arising in the context, whereas $A
\asymp B$ means $A\lesssim B$, and $B\lesssim A$.  The notation $B_1
\hookrightarrow B_2$ denotes
the continuous embedding of
the space $B_1$ into $B_2$.

\section{Preliminary results}\label{sec2}
In this section we recall the
definition and some
properties of modulation and Wiener amalgam
spaces, which will be used
later on. We refer to
\cite{F1,
Feichtinger-grochenig89,grochenig,triebel83}
for the general theory.\par
We consider the functions
$\la\xi\ra^s:=(1+|\xi|^2)^{s/2}$,
$s\in\R$, as weight
functions. Then, weighted
modulation spaces are defined
in terms of the following time-frequency representation: the
short-time Fourier transform
(STFT) $V_gf$ of a
function/tempered
distribution $f\in \mathcal{S}'(\R^d)$
with respect to the window
$g\in \mathcal{S}(\R^d)\setminus\{0\}$ is defined by
\begin{equation}\label{STFT}
V_g f(x,\xi)=\la f,M_\xi T_xg\ra= \int_{\R^d}
e^{-2\pi i \xi
y}f(y)\overline{g(y-x)}\,dy,
\end{equation}
i.e.\ the  Fourier transform of $f\overline{T_x g}$.

Given a window
$g\in \mathcal{S}(\R^d)\setminus\{0\}$, $1\leq
p,q\leq \infty$ and
$\delta\in\R$, the {\it
  modulation space} $M^{p,q}_\delta(\R^d)$,
   consists of all tempered
distributions
$f\in \mathcal{S}'(\R^d)$ such that
$V_gf\in
L^{p,q}_{1\otimes\langle\cdot\rangle^\delta}(\R^{2d})$
(weighted mixed-norm Lebesgue space).
The norm on $M^{p,q}_\delta$
is therefore defined as
$$
\|f\|_{M^{p,q}_\delta}=\|V_gf\|_{L^{p,q}_{1\otimes\langle\cdot\rangle^\delta}}=\left(\int_{\R^d}
  \left(\int_{\R^d}|V_gf(x,\xi)|^p \langle\xi\rangle^
  {\delta p}\,
    dx\right)^{q/p}d\xi\right)^{1/p}
$$
(with obvious changes when
$p=\infty$ or $q=\infty$). If
$p=q$, we write $M^p_\delta$
instead of $M^{p,p}_\delta$,
and if $\delta=0$ we write
$M^{p,q}$ and $M^p$ for
$M^{p,q}_0$ and $M^{p,p}_0$,
respectively. Then
$M^{p,q}_\delta(\R^d)$ is a Banach
space and different windows $g\in \mathcal{S}(\R^d)\setminus\{0\}$ give equivalent norms. For the
properties of these spaces we
refer to the literature
quoted at the beginning of
this subsection.\par
For $1\leq p,q\leq\infty$ the Wiener amalgam space $W(\CF L^p,L^q)(\R^d)$ consists of the temperate distributions $f\in\mathcal{S}'(\R^d)$ such that 
\[
\|f\|_{W(\CF L^p,L^q)}=\Big(\int_{\R^d}\|g(\cdot-x)f\|^q_{\CF L^p}\, dx\Big)^{1/q}<\infty,
\]
where $g\in \mathcal{S}(\R^d)\setminus\{0\}$ is an arbitrary window (with obvious changes if $q=\infty$). It is easy to show that  the Fourier transform establishes an isomorphism 
\[
\CF: M^{p,q}(\R^d) \to W(\CF L^p, L^{q})(\R^d).
\]

The duality theory goes as expected, namely
 $(M^{p,q}_\delta)^*=M^{p',q'}_{-\delta}$, with $1\leq
 p,q<\infty$, $p',q'$ being the conjugate
 exponents.\par
Here is the basic complex interpolation
result (see e.g.
\cite{F2,F1,Feichtinger-grochenig89} and \cite[Theorem 2.3]{baoxiang3} for a direct proof).
\begin{proposition}\label{interp} Let $0<\theta<1$,
 $p_j,q_j\in [1,\infty]$ and $\delta_j\in\R$, $j=1,2$.
Set
$$ \frac1p=\frac{1-\theta}{p_1}+\frac{\theta}{p_2},\quad \frac1q=\frac{1-\theta}{q_1}+\frac{\theta}{q_2},\quad \delta=(1-\theta)\delta_1+
{\theta}\delta_2.$$ Then
$$(M^{p_1,q_1}_{\delta_1}(\R^d),M^{p_2,q_2}_{\delta_2}(\R^d))_{[\theta]}=M^{p,q}_{\delta}(\R^d).$$
\end{proposition}

We now recall the dilation
properties. For $(1/p,1/q)\in
[0,1]\times [0,1]$, we define
the subsets
$$ I_1:\ \max (1/p,1/p')\leq 1/q,\quad\quad I_1^*:\ \min (1/p,1/p')\geq 1/q,
$$
$$ I_2:\ \max (1/q,1/2)\leq 1/p',\quad\quad I_2^*:\ \min (1/q,1/2)\geq  1/p',
$$
$$ I_3:\ \max (1/q,1/2)\leq 1/p,\quad\quad I_3^*:\ \min (1/q,1/2)\geq
1/p,
$$
as shown in Figure 1 below. 
\begin{figure}
           \includegraphics[width=12truecm]{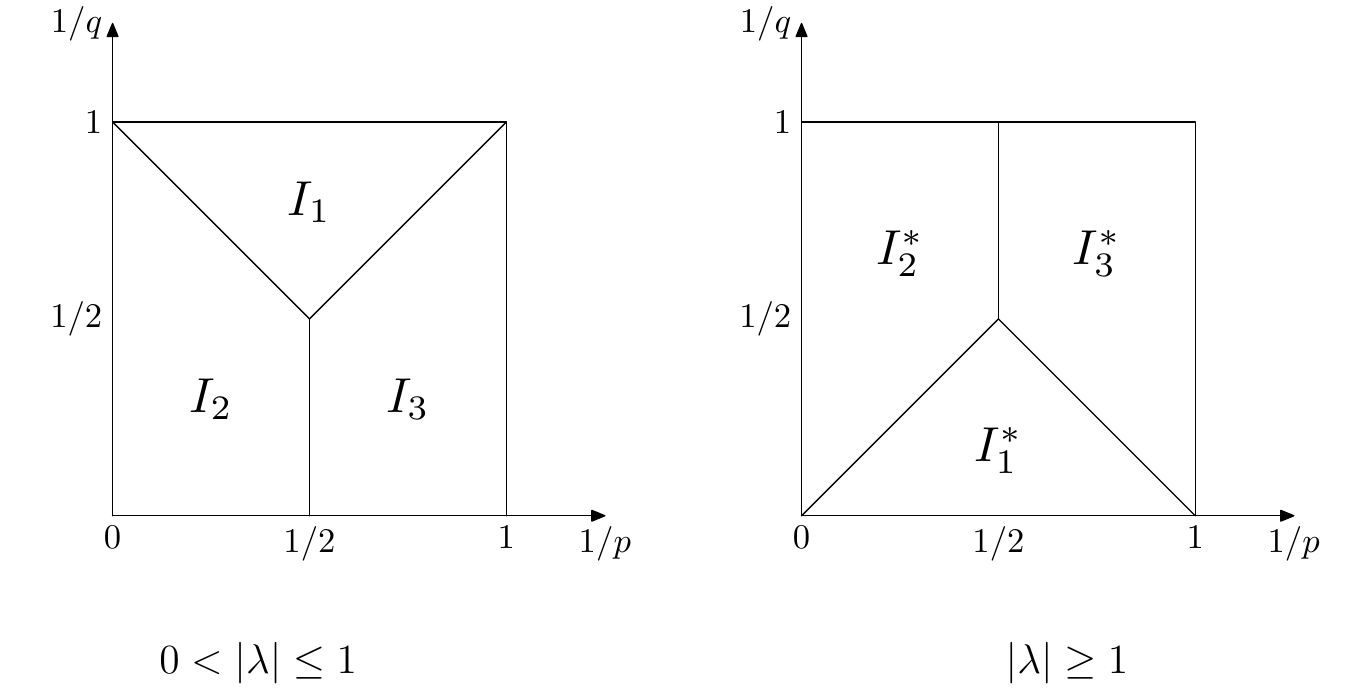}
\caption{The index sets}           
\end{figure}

We introduce the indices:
\begin{equation}\label{iin1}
 \mu_1(p,q)=\begin{cases}-1/p &  \quad {\mbox{if}}
 \quad (1/p,1/q)\in  I_1^*,\\
 1/q-1 &   \quad {\mbox{if}}  \quad (1/p,1/q)\in  I_2^*,\\
 -2/p +1/q&  \quad  {\mbox{if}}
 \quad (1/p,1/q)\in  I_3^*,\\
 \end{cases}
\end{equation}
and
\begin{equation}\label{iin2} \mu_2(p,q)=\begin{cases}-1/p &  \quad {\mbox{if}} \quad (1/p,1/q)\in  I_1,\\
 1/q-1 &   \quad {\mbox{if}}  \quad (1/p,1/q)\in  I_2,\\
 -2/p +1/q&  \quad  {\mbox{if}}  \quad (1/p,1/q)\in  I_3.\\
 \end{cases}
 \end{equation}
Here is the main result about
the behaviour of the dilation
operator in modulation spaces. Set
$U_\lambda f(x):=f(\lambda
x)$, $\lambda\not=0$.
\begin{theorem}{\cite[Theorem
3.1]{sugimoto-tomita}}\label{dilprop}
Let $1\leq p,q\leq\infty,$ and
$\lambda\not=0$.\par\medskip
 (i) We have
$$\| U_\lambda f\|_{M^{p,q}}\lesssim |\lambda|^{d\mu_1(p,q)}
\|f\|_{M^{p,q}},\quad\quad\forall\
|\lambda|\geq 1,\ \forall
f\in M^{p,q}(\R^d).
$$
Conversely, if there exists $\a\in\R$
such that
$$\| U_\lambda f\|_{M^{p,q}}\lesssim |\lambda|^{\a}\|f\|_{M^{p,q}},
\quad\quad \forall\ |\lambda|\geq 1,\
\forall f\in M^{p,q}(\R^d),
$$
then $\a\geq d\mu_1(p,q)$.\par\medskip
 (ii) We
have
$$\| U_\lambda f\|_{M^{p,q}}\lesssim |\lambda|^{d\mu_2(p,q)}
\|f\|_{M^{p,q}},\quad\quad\forall\
0<|\lambda|\leq 1,\ \forall
f\in M^{p,q}(\R^d).
$$
Conversely, if there exists $\b\in\R$
such that
$$\| U_\lambda f\|_{M^{p,q}}\lesssim |\lambda|^{\b}\|f\|_{M^{p,q}},
\quad\quad \forall\ 0<|\lambda|\leq 1,\
\forall f\in M^{p,q}(\R^d),
$$
then $\b\leq d\mu_2(p,q)$.
\end{theorem}

By a conjugation with the Fourier transform one deduces the following dilation property for Wiener amalgam spaces.

\begin{corollary}{\cite[Corollary 3.2]{Cordero_2008_Metaplectic}}\label{cor2.3}
With the above notation,
\[
\|U_\lambda f\|_{W(\CF
L^1,L^\infty)}\lesssim \|f\|_{W(\CF
L^1,L^\infty)}\quad \forall\ 0<|\lambda|\leq 1,\forall f\in W(\CF L^1,L^\infty)(\R^d).
\]
\end{corollary}

The following proposition can be deduced, via Fourier transform, from the convolution properties of modulation spaces. 

\begin{proposition}{\cite[Proposition 2.5]{Cordero_2008_Metaplectic}}\label{prop2.4}
For every $1\leq
p,q\leq\infty$ we have
\[
\|fu\|_{W(\CF
L^p,L^q)}\lesssim
\|f\|_{W(\CF
L^1,L^\infty)}\|u\|_{W(\CF
L^p,L^q)}.
\]
\end{proposition}

We conclude this preliminary section with a known result (see e.g.\ \cite[Lemma 8]{benyi}), which we recall together with a shorter proof for the benefit of the reader.

\begin{lemma}\label{lem:s_in_W_bounded}
Let $\s \in W(\CF L^1, L^{\infty})$. Then,
$$ \s(D): M^{p,q} \to M^{p,q}$$ is bounded, for every $1 \leq p, q \leq \infty$.
\end{lemma}
\begin{proof}
We can write $ \s(D)= \CF^{-1}\circ A_{\s}\circ \CF$, where $ A_{\s}f(\xi)= \s(\xi)f(\xi)$. Using Proposition \ref{prop2.4} we have
\begin{align*}
    \|A_{\s}f\|_{W(\CF L^p, L^{q})}&= \| \s f\|_{W(\CF L^p, L^{q})}\\
    &\lesssim \| \s\|_{W(\CF L^1, L^{\infty})} \|f\|_{W(\CF L^p, L^{q})},
\end{align*} 
so that $A_{\s}:W(\CF L^p, L^{q}) \to W(\CF L^p, L^{q})$ is bounded, for every $1 \leq p, q \leq \infty$. Hence, since the Fourier transform establishes an isomorphism $\CF: M^{p,q} \to W(\CF L^p, L^{q})$, 
we see that $ \s(D): M^{p,q} \to M^{p,q}$ is bounded too.
\end{proof}

\section{No growth, strong oscillations}\label{sec3}
This section is devoted to the proof of Theorem \ref{lem:A_bound}. We begin with a preliminary result which is strongly inspired by \cite[Lemmas 2.1 and 2.2]{Concetti_2009_Schatten}, where a similar investigation is carried on in the framework of modulation spaces (as opposite to the Wiener amalgam spaces considered here).
\begin{lemma}
Let $f \in W(\CF L^1, L^{\infty})(\R^d)$ and $\chi \in C_0^{\infty}(B)$, where $B$ is an open ball with center at the origin. Let $$ g_{x_0}(x)= \chi(x-x_{0})\int_{0}^1(1-t)f(t(x-x_0)+ x_0)\, dt,$$ for some $x_0 \in \R^d$.\par
 Then $g_{x_0} \in W(\CF L^1, L^{\infty})(\R^d)$, and for some constant $C$ independent of $x_0$ and $f$ we have
$$\|g_{x_0}\|_{W(\CF L^1, L^{\infty})} \leq C \|f\|_{W(\CF L^1, L^{\infty})}.$$
\end{lemma}
\begin{proof}
Using Proposition \ref{prop2.4} and Corollary \ref{cor2.3} we have
\begin{align*}
    \| &g_{x_0}(x) \|_{W(\CF L^1, L^{\infty})} =  \left\| \chi(x-x_{0})\int_{0}^1(1-t)f(t(x-x_0)+ x_0)dt \right\|_{W(\CF L^1, L^{\infty})}\\
    &\lesssim \left\| \chi(x-x_{0})\right\|_{W(\CF L^1, L^{\infty})}\left\|\int_{0}^1(1-t)f(tx+(1-t) x_0)dt \right\|_{W(\CF L^1, L^{\infty})}\\
     &\leq  \left\| \chi(x-x_{0})\right\|_{W(\CF L^1, L^{\infty})}\int_{0}^1(1-t)\left\|f(tx+(1-t) x_0)\right\|_{W(\CF L^1, L^{\infty})}dt \\
     &=  \left\| \chi\right\|_{W(\CF L^1, L^{\infty})}\int_{0}^1(1-t)\left\|f(tx)\right\|_{W(\CF L^1, L^{\infty})}dt \\
     &\lesssim  \left\| \chi\right\|_{W(\CF L^1, L^{\infty})}\int_{0}^1(1-t)\left\|f\right\|_{W(\CF L^1, L^{\infty})}dt \\
    & \lesssim \left\|f\right\|_{W(\CF L^1, L^{\infty})}.
\end{align*}
\end{proof}
\begin{lemma}\label{lem:bounded_by_partial}
Assume that $B \subset\R^n$ is an open ball, $\mu \in C^2(\R^d)$ is real-valued and satisfies $\partial^{\g}\mu \in W(\CF L^1, L^{\infty})(\R^d)$ for all multi-indices $\g$ with $|\g|=2$ and that $f\in M^1(\R^d) \cap \CE'(B)$. Then $f e^{i \mu}\in M^1(\R^d)$ and for some constant $C$  which only depends on $d$ and the radius of the ball $B$ we have
$$
\|f e^{i \mu}\|_{M^1} \leq C \|f\|_{M^1} \exp \Big(C \sum_{|\g|=2}\|\partial^{\g} \mu \|_{W(\CF L^1, L^{\infty})}\Big).
$$
\end{lemma}
\begin{proof}
We may assume that $B$ is the unit ball which is centered at the origin. By Taylor expansion it follows that $\mu = \psi_1 + \psi_2$, where 
$$
\psi_1(x)= \mu(0)+ \langle \nabla \mu (0), x \rangle , \ \ \psi_2(x)= \sum_{|\g|=2}\frac{2}{\g !} \int_0^1 (1-t)\partial^{\g} \mu(t x)dt \, x^{\g}.
$$
Since modulations do not affect the modulation space norms we have $\|f e^{i \psi_1}\|_{M^1}= \|f \|_{M^1}$. Furthermore, if $\chi \in C_0^{\infty}(\R^d)$ satisfies $\chi (x)= 1$ on $B$, then it follows from the previous Lemma that, for some constant $C_1>0$,
$$\|\chi \psi_2\|_{W(\CF L^1, L^{\infty})}\leq C_1 \sum_{|\g|=2}\|\partial^{\g}\mu\|_{W(\CF L^1, L^{\infty})}.$$
Hence, by Proposition \ref{prop2.4}, for some $C_2\geq1$ we have
\begin{align*}
    \|e^{i \chi \psi_2}\|_{W(\CF L^1, L^{\infty})}&=  \left\|\sum_{n=0}^{\infty} \frac{(\chi \psi_2)^n}{n!}\right\|_{W(\CF L^1, L^{\infty})} 
    \leq \sum_{n=0}^{\infty} \frac{C_2^{n-1}}{n!}\|\chi \psi_2\|_{W(\CF L^1, L^{\infty})}^n \\
    &\leq \exp\Big(C_2\|\chi \psi_2\|_{W(\CF L^1, L^{\infty})}\Big) \\
&    \leq \exp\Big( C_1C_2 \sum_{|\g|=2}\|\partial^{\g}\mu\|_{W(\CF L^1, L^{\infty})} \Big).
\end{align*} 
Using $M^1=W(\CF L^1,L^\infty)$ and Proposition \ref{prop2.4} again, this gives
\begin{align*}
    \|f e^{i \mu}\|_{M^1} &= \|f e^{i \psi_1}e^{i \chi \psi_2}\|_{M^1} \lesssim\|f e^{i \psi_1}\|_{M^1} \|e^{i \chi \psi_2}\|_{W(\CF L^1, L^{\infty})} \\
   & \leq C \|f\|_{M^1} \exp\Big( C \sum_{|\g|=2}\|\partial^{\g}\mu\|_{W(\CF L^1, L^{\infty})} \Big).
\end{align*}
\end{proof}
\begin{proof}[Proof of Theorem \ref{lem:A_bound}]
Let us first show that $e^{i \mu (x)} \in W(\CF L^1, L^{\infty})$. We know that there exists $\chi \in C_0^{\infty}(\R^d)$ (cf.\ \cite{F2,F1}) such that
\[
    \|e^{i \mu (x)}\|_{W(\CF L^1, L^{\infty})} = \sup_{k \in \Z^d} \{ \|\chi(x-k)e^{i \mu(x)}\|_{\CF L^1} \} \asymp \sup_{k \in \Z^d}\{ \|\chi(x-k)e^{i \mu(x)}\|_{M^1}\} 
    \]
 where the last equivalence follows from that fact the for functions supported in a ball the $\CF L^1$ and $M^1$ norms are equivalent, with constants depending only on the radius of the ball. \par
 Hence, using Lemma \ref{lem:bounded_by_partial} we can continue our estimate as 
    \begin{align*}
    &\leq \sup_{k \in \Z^d} \Big\{C \|\chi(x-k)\|_{M^1} \exp\Big(C \sum_{|\g|=2}\|\partial^{\g}\mu\|_{W(\CF L^1, L^{\infty})} \Big) \Big\}\\
    &= C\|\chi\|_{M^1}  \exp\Big(C \sum_{|\g|=2}\|\partial^{\g}\mu\|_{W(\CF L^1, L^{\infty})} \Big).
\end{align*}

Hence $e^{i \mu (x)}\in W(\CF L^1, L^{\infty})$ and by Lemma \ref{lem:s_in_W_bounded} we deduce that $ e^{i \mu (D)}: M^{p,q} \to M^{p,q}$ is bounded, for every $1 \leq p, q \leq \infty$. 
\end{proof}

\section{Growth at infinity, strong oscillations}\label{sec4}
In this section we are going to prove Theorem \ref{caso_pq}. 

We begin with the following preliminary result.
\begin{lemma}\label{lem:partial_mu_W}
Let $\mu(\xi)$ be a real-valued $C^2$ function, satisfying 
 \begin{equation}\label{eq:cond_mu}
    \langle \xi \rangle^{2-\a} \partial^{\g} \mu \in W(\CF L^1, L^{\infty})(\R^d) \ \text{for } |\g|=2.
\end{equation}
Then 
\begin{itemize}
    \item[(i)]$ \langle \xi \rangle^{-\a}  \mu \in W(\CF L^1, L^{\infty})(\R^d)$,\\
    \item[(ii)] $ \langle \xi \rangle^{1-\a} \partial^{\g} \mu \in W(\CF L^1, L^{\infty})(\R^d) \ \text{for } |\g|=1.$
\end{itemize}
\end{lemma}
\begin{proof}
To prove (i), consider a Taylor expansion
$$
\mu(\xi)= \mu (0) + \langle \nabla\mu (0), \xi \rangle + \sum_{|\g|=2}\frac{2}{\g !} \int_0^1 (1-t)\partial^{\g} \mu(t \xi)dt \,\xi^{\g}.
$$
Hence 
\begin{multline}\label{eq:taylor_mu}
    \langle \xi \rangle^{-\a} \mu(\xi)= \mu (0)\langle \xi \rangle^{-\a}  + \langle \nabla\mu (0), \xi \rangle \langle \xi \rangle^{-\a} \\
    + \sum_{|\g|=2}\frac{2}{\g !} \int_0^1 (1-t)\partial^{\g} \mu(t \xi)dt \, \xi^{\g}\langle \xi \rangle^{-\a} 
\end{multline}
where \[
\mu (0)\langle \xi \rangle^{-\a}\in W(\CF L^1, L^{\infty})(\R^d),\quad \langle \nabla\mu (0), \xi \rangle \langle \xi \rangle^{-\a} \in W(\CF L^1, L^{\infty})(\R^d),
\]
because $\a \geq 2$. Here we used that fact that the functions $\langle \xi \rangle^{-\a}$ and $\xi_j \langle \xi \rangle^{-\a}$ are bounded together their derivatives of every order, so that they belong to $M^{\infty,1}(\R^d)$ (\cite[Theorem 14.5.3]{grochenig}) and hence to $W(\CF L^1, L^{\infty})(\R^d)$ as well.\par
 Let us show that the last summation in \eqref{eq:taylor_mu} belongs to $W(\CF L^1, L^{\infty})(\R^d)$ too. We have 
 \begin{align*}
&\left\|\int_0^1 (1-t)\partial^{\g} \mu(t \xi)dt \xi^{\g}\langle \xi \rangle^{-\a} \right\|_{ W(\CF L^1, L^{\infty})}\\
&=\left \|\int_0^1 (1-t)\partial^{\g} \mu(t \xi) \langle t\xi \rangle^{2-\a} \langle t\xi \rangle^{-2+\a} dt\, \xi^{\g}\langle \xi \rangle^{-\a} \right \|_{ W(\CF L^1, L^{\infty})} \\ 
&\lesssim \left \|\int_0^1 (1-t)\langle t \rangle^{-2+\a}\partial^{\g} \mu(t \xi) \langle t\xi \rangle^{2-\a} dt\, \xi^{\g}\langle \xi \rangle^{-\a}  \langle \xi \rangle^{-2+\a} \right \|_{ W(\CF L^1, L^{\infty})}. 
\end{align*}
Using Proposition \ref{prop2.4} and Corollary \ref{cor2.3} we can continue the above estimate as
\begin{align*} 
&\lesssim \int_0^1 (1-t)\langle t \rangle^{-2+\a}\left \|\partial^{\g} \mu(t \xi) \langle t\xi \rangle^{2-\a} \right \|_{ W(\CF L^1, L^{\infty})} dt \left \| \xi^{\g}\langle \xi \rangle^{-2} \right \|_{ W(\CF L^1, L^{\infty})}\\
&\lesssim \int_0^1 (1-t)\langle t \rangle^{-2+\a}dt \left \|\partial^{\g} \mu( \xi) \langle \xi \rangle^{2-\a} \right \|_{ W(\CF L^1, L^{\infty})}  \left \| \xi^{\g}\langle \xi \rangle^{-2} \right \|_{ W(\CF L^1, L^{\infty})} \\
&\lesssim \left \|\partial^{\g} \mu( \xi) \langle \xi \rangle^{2-\a} \right \|_{ W(\CF L^1, L^{\infty})}  \left \| \xi^{\g}\langle \xi \rangle^{-2} \right \|_{ W(\CF L^1, L^{\infty})}.
\end{align*}
This concludes the proof of (i) because, arguing as above, we have $\xi^{\g}\langle \xi \rangle^{-2} \in M^{\infty,1}\subset W(\CF L^1, L^{\infty})$, whereas $\partial^{\g} \mu( \xi) \langle \xi \rangle^{2-\a} \in  W(\CF L^1, L^{\infty})$ by assumption.\par
To prove (ii), consider the Taylor expansion of $\partial^{\g}\mu$, for $|\g|=1$
$$
\partial^{\g}\mu(\xi)= \partial^{\g}\mu (0) + \sum_{|\b|=1} \int_0^1 \partial^{\g+ \b} \mu(t \xi)dt\, \xi^{\b},
$$
so that
$$
\langle \xi \rangle^{1-\a} \partial^{\g}\mu(\xi)= \partial^{\g}\mu (0)\langle \xi \rangle^{1-\a}  + \sum_{|\b|=1} \int_0^1 \partial^{\g+ \b} \mu(t \xi)dt\,\xi^{\b}\langle \xi \rangle^{1-\a}.
$$
Now $ \partial^{\g}\mu (0)\langle \xi \rangle^{1-\a} \in W(\CF L^1, L^{\infty})$, because $\a\geq 2$, and arguing as above
\begin{align*}
&\left\|\int_0^1 \partial^{\g+ \b} \mu(t \xi)dt\, \xi^{\b}\langle \xi \rangle^{1-\a} \right\|_{ W(\CF L^1, L^{\infty})}\\
&= \left \|\int_0^1 \partial^{\g+ \b} \mu(t \xi) \langle t\xi \rangle^{2-\a} \langle t\xi \rangle^{-2+\a} dt\, \xi^{\b}\langle \xi \rangle^{1-\a} \right \|_{ W(\CF L^1, L^{\infty})} \\ 
&\lesssim \left \|\partial^{\g+\b} \mu( \xi) \langle \xi \rangle^{2-\a} \right \|_{ W(\CF L^1, L^{\infty})}  \left \| \xi^{\b}\langle \xi \rangle^{-1} \right \|_{ W(\CF L^1, L^{\infty})} ,
\end{align*}
where $\xi^{\b}\langle \xi \rangle^{-1} \in M^{\infty,1}(\R^d)\subset W(\CF L^1, L^{\infty})(\R^d)$ because $|\beta|=1$, and moreover  $\partial^{\g + \b} \mu( \xi) \langle \xi \rangle^{2-\a} \in  W(\CF L^1, L^{\infty})(\R^d)$ by assumption, because $\g + \b=2$.
\end{proof}

We observe that, by complex interpolation of weighted modulation spaces, namely Proposition \ref{interp}, it suffices to prove the  conclusion of Theorem \ref{caso_pq} when $(p,q)$ is one of the four vertices of the
interpolation square,
$(1,1)$, $(1,\infty)$,
$(\infty,1)$,
$(\infty,\infty)$, with
$\delta=d(\alpha-2)/2$,
as well as for the points
$(2,1)$, $(2,\infty)$ with
$\delta=0$. To this end, we
reduce matters to the case of
unweighted modulation spaces
by means of the following
easy lemma.
\begin{lemma}\label{lem1f}
A multiplier $\sigma(D)$ is
bounded from
$\cM^{p,q}_{\delta}(\R^d)$ to
$\cM^{p,q}(\R^d)$ if and only if
the multiplier
$\sigma(D)\langle D
\rangle^{-\delta}$ is bounded
on $\cM^{p,q}(\R^d)$.
\end{lemma}
\begin{proof}
We know e.g.\ from
\cite[Theorem 2.2, Corollary
2.3]{Toftweight} that
$\langle D\rangle^{t}$
defines an isomorphism
$\cM^{p,q}_s(\R^d)\to
\cM^{p,q}_{s-t}(\R^d)$ for every
$s,t\in\mathbb{R}$, so that
the conclusion is immediate.
\end{proof}

Therefore we may work with the operator
$$Tf(x)=\int_{\R^d} e^{2\pi ix\xi}e^{i\mu(\xi)}\langle\xi
\rangle^{-\delta}\hat{f}(\xi)d\xi.$$
We have to prove that $T$ is
 bounded on
$M^{1,1}$,
$M^{1,\infty}$,
$\cM^{\infty,1}$,
$\cM^{\infty,\infty}$ for
$\delta=\frac{d(\alpha-2)}{2}$,
and on $M^{2,1}$ and
$\cM^{2,\infty}$ for
$\delta=0$.\par\bigskip
\noindent \textbf{Boundedness
on $M^{1,1}$ and
$\cM^{\infty,1}$
 for $\delta=\frac{d(\alpha-2)}{2}$}.
We will need the following
lemma (cf.
\cite{cordero-nicola-rodino,sugimoto-tomita}).
\begin{lemma}\label{sum_1}
Let $\chi$ be a smooth
function supported where
$B_0^{-1}\leq |\xi| \leq B_0$
for some $B_0>0$. Then, for
$1\leq p\leq\infty$,
$$ \sum_{j=1}^{\infty} \|\chi(2^{-j}D)f\|_{M^{p,1}}\leq C\|f\|_{M^{p,1}}.$$
\end{lemma}
\begin{proof}
We will use the following
characterization of the
$M^{p,q}$ norm
\cite{triebel83}: let
$\varphi \in
C^{\infty}_0(\R^d)$ such that
$\varphi(\xi)\geq 0$,
$\sum_{m\in\Z^d} \varphi
(\xi-m)=1$, $\forall \xi\in
\mathbb{R}^d$. Then
$$\|f\|_{M^{p,q}}\asymp \Big(\sum_{m\in\Z^d} \|\varphi
(D-m)f\|^q_{L^p}\Big)^{1/q}.$$
Hence it turns out
\begin{align*}
\sum_{j=1}^{\infty}\|\chi(2^{-j}D)f\|_{M^{p,1}}&\asymp
\sum_{j=1}^{\infty}\sum_{m\in\Z^d}
\|\varphi
(D-m)\chi(2^{-j}D)f\|_{L^p}
\\
&= \sum_{m\in\Z^d}
\sum_{j=1}^{\infty}\|\varphi
(D-m)\chi(2^{-j}D)f\|_{L^p}\\
& =\sum_{m\in\Z^d}
\sum_{j=1}^{\infty}\|\chi(2^{-j}D)\varphi
(D-m)f\|_{L^p}.
\end{align*}
Now, the number of indices
$j\geq1$ for which ${\rm
supp}\,\chi(2^{-j}\cdot)\cap
{\rm supp}\,\varphi
(\cdot-m)\neq 0$ is finite
for every $m$, and even
uniformly bounded with
respect to $m$. Hence the
last expression is
$$\lesssim\sum_{m\in\Z^d} \sup_{j\geq1}\|\chi(2^{-j}D)\varphi
(D-m)f\|_{L^p}.$$ Since the
operators $\chi(2^{-j}D)$ are
uniformly bounded on $L^p$ we
can continue the estimate as
$$\lesssim\sum_{m\in\Z^d} \|\varphi (D-m)f\|_{L^p}\asymp \|f\|_{M^{p,1}}.$$
\end{proof}

Consider now a
Littlewood-Paley
decomposition of the
frequency domain. Namely,
 fix a smooth function $\psi_0$
  such that $\psi_0(\xi)=1$
  for $|\xi|\leq1$ and
  $\psi_0(\xi)=0$ for
  $|\xi|\geq2$.
Set
$\psi(\xi)=\psi_0(\xi)-\psi_0(2\xi)$.
Then $\psi_j(\xi):=\psi(2^{-j}\xi)$ for
$j\geq1$ is supported where
$2^{j-1}\leq|\xi| \leq 2^{j+1}$. We can write
\begin{equation}\label{decom1}
T=T^{(0)}+\sum_{j=1}^{\infty}
T^{(j)} \end{equation}
 where
$T^{(j)}$ is the Fourier
multiplier with symbol
$\sigma_j(\xi):=
e^{i\mu(\xi)}\psi_j(\xi)
\langle\xi
\rangle^{-\delta}$, $j\geq0$.\par
Now,  $T^{(0)}$ is
bounded on
$\cM^{p,q}$ for every $1\leq
p,q\leq\infty$ as a consequence of Lemma \ref{lem:s_in_W_bounded}, because $\sigma_0\in M^1\subset W(FL^1,L^\infty)$ by Lemma \ref{lem:bounded_by_partial}.\par
Consider now the above
sum over $j\geq1$. \par
 Let
$$\lambda_j =2^{-\frac{\alpha-2}{2}j},$$
and consider the operators
$\tilde{T}^{(j)}$ defined by
\begin{equation}\label{decom2}
T^{(j)}=U_{\lambda_j}\tilde{T}^{(j)}U_{\lambda_j^{-1}},
\end{equation}
where $U_\lambda
f(x)=f(\lambda x)$,
$\lambda>0$, is the dilation
operator.
 In other terms,
\[
    \tilde{T}^{(j)}f(x)= \int_{\R^d}e^{2\pi i x \xi}e^{i \mu (\l_j \xi)}\psi_j(\l_j \xi)\langle \l_j \xi\rangle^{-\delta}\hat{f}(\xi)d \xi.
    \]
    
 Let $\chi_j(\chi):= \chi(2^{-j}\xi)$ with $\chi\in C^\infty_0(\R^d)$ supported where $\frac{1}{4}\leq |\xi| \leq 4$ and $\chi(\xi)= 1$ on the support of $\psi$, so that $\chi_j(\xi)=1$ on the support of $\psi_j$. 
 We can therefore write
  \[
   \tilde{T}^{(j)}f(x) = \int_{\R^d}e^{2\pi i x \xi}e^{i \chi_j(\l_j \xi) \mu (\l_j \xi)}\psi_j(\l_j \xi)\langle \l_j \xi\rangle^{-\delta}\hat{f}(\xi)d \xi,  
\] 
hence 
\begin{equation}\label{add0}
\tilde{T}^{(j)}= A_j B_j,
\end{equation}
 where
$$
A_j= e^{i (\chi_j \mu) (\l_j D)}, \ \ \ \ B_j= \psi_j(\l_j D)\langle \l_j D\rangle^{-\delta}.
$$

 Taking into account that on the support of $\psi_j(\l_j \xi)$ we have $\l_j|\xi| \asymp 2^j$ and $\delta=d(\alpha-2)/2$, the following estimate is easily verified:
\[
|\partial^\gamma(\psi_j(\lambda_j\xi)
\langle \lambda_j\xi
\rangle^{-\delta})|\lesssim
2^{-\frac{d(\alpha-2)}{2}j},\quad
\forall\gamma\in\mathbb{Z}^d_+.
\]

Then, by the classical boundedness results of pseudodifferential operators on modulation spaces (see e.g.\ \cite[Theorems 14.5.2, 14.5.2]{grochenig}) 
 we have 
 \begin{equation}\label{add1}
  \|B_j\|_{M^{p,q} \to M^{p,q}} \lesssim 2^{-\frac{d(\alpha-2)}{2}j},
  \end{equation}
   for every $1\leq p,q \leq \infty$.

Let us now prove that 
\begin{equation}\label{add2}
\|A_j\|_{M^{p,q} \to M^{p,q}} \lesssim 1,
\end{equation}
 for all $j \geq 1$ and for every $1\leq p,q \leq \infty$.

Using Theorem \ref{lem:A_bound} it is sufficient to check that 
$$
\|\partial^{\g} [\chi_j(\l_j \xi)\mu(\l_j \xi)]\|_{W(\CF L^1, L^{\infty})}\lesssim 1, $$for $|\g|=2$ and all $j \geq 1$ (we are in fact using the fact that the operator norm of the multiplier in Theorem \ref{lem:A_bound} is bounded when $ \partial^{\g} \mu$, $|\gamma|=2$, belong to a bounded subset of $W(\CF L^1,L^\infty)$).\par 
For $|\g|=2$,  we have
$$\partial^{\g} [\chi_j(\l_j \xi)\mu(\l_j \xi)]= \l_j^2 \partial^{\g} [\chi_j\mu](\l_j \xi),$$
and by Leibniz' formula it is enough to prove that 
\begin{eqnarray}
\l_j^2 \|(\partial^{\g}\chi_j)\mu\|_{W(\CF L^1, L^{\infty})}\lesssim 1 & |\g|=2 \label{eq:partial_chi}\\
\l_j^2 \|\partial^{\g}\chi_j \partial^{\b}\mu\|_{W(\CF L^1, L^{\infty})}\lesssim 1 & \ \qquad |\g|=|\b|=1 \label{eq:partial_chi_mu}\\
\l_j^2 \|\chi_j \partial^{\g}\mu\|_{W(\CF L^1, L^{\infty})}\lesssim 1 & |\g|=2 \label{eq:partial_mu}.
\end{eqnarray}
First, let us prove \eqref{eq:partial_chi}. Using Lemma \ref{lem:partial_mu_W} (i), Proposition \ref{prop2.4} and the embeddings $C^{d+1}(\R^d)\hookrightarrow M^{\infty,1}(\R^d)\hookrightarrow W(\CF L^1,L^\infty)(\R^d)$ (\cite[Theorem 14.5.3]{grochenig}) we can estimate
\begin{align*}
    \l_j^2 \|(\partial^{\g}\chi_j) \mu\|_{W(\CF L^1, L^{\infty})}&\lesssim \l_j^2 \|\langle \xi \rangle^{-\a} \mu\|_{W(\CF L^1, L^{\infty})} \|\langle \xi \rangle^{\a} \partial^{\g}\chi_j \|_{W(\CF L^1, L^{\infty})}\\
    &\lesssim \l_j^2 \sum_{\b \leq d+1}\|\partial^{\b}[\langle \xi \rangle^{\a} \partial^{\g}\chi_j]\|_{L^{\infty}}.
\end{align*}
On the other hand,
\begin{align*}
   | \partial^{\b}[\langle \xi \rangle^{\a} \partial^{\g}\chi_j(\xi)]|&= \left|\sum_{\nu \leq \b} {{\b}\choose{\nu}}\partial^{\nu} \langle \xi \rangle^{\a} \partial^{\g+ \b-\nu}\chi_j(\xi) \right| \\
 &  \lesssim \sum_{\nu \leq \b} {{\b}\choose{\nu}} \langle \xi \rangle^{\a-|\nu|} 2^{-j|\g+ \b-\nu|}\left|(\partial^{\g+ \b-\nu}\chi)(2^{-j}\xi) \right|\\
   & \lesssim \sum_{\nu \leq \b} {{\b}\choose{\nu}}2^{(\a-|\nu|) j} 2^{-2j} \lesssim 2^{j(\a -2)},
\end{align*}
because on the support of $\chi_j$, $|\xi| \asymp 2^j$ and $|\g+ \b-\nu|\geq 2$. Thus
\[
    \l_j^2 \|(\partial^{\g}\chi_j) \mu\|_{W(\CF L^1, L^{\infty})}
    \lesssim \l_j^2 2^{j(\a -2)} =1.
\]

Now, let us prove \eqref{eq:partial_chi_mu}, using Lemma \ref{lem:partial_mu_W} (ii) and arguing as above we write
\begin{align*}
    \l_j^2 \|\partial^{\g}\chi_j \partial^{\b}\mu\|_{W(\CF L^1, L^{\infty})}
    &\lesssim \l_j^2 \|\langle \xi \rangle^{1-\a} \partial^{\b}\mu\|_{W(\CF L^1, L^{\infty})} \|\langle \xi \rangle^{\a-1} \partial^{\g}\chi_j \|_{W(\CF L^1, L^{\infty})}\\
    &\lesssim \l_j^2 \sum_{\b \leq d+1}\|\partial^{\b}[\langle \xi \rangle^{\a-1} \partial^{\g}\chi_j]\|_{L^{\infty}}.
\end{align*}
On the other hand,
\begin{align*}
   | \partial^{\b}[\langle \xi \rangle^{\a-1} \partial^{\g}\chi_j(\xi)]|&
   \lesssim \sum_{\nu \leq \b} {{\b}\choose{\nu}} \langle \xi \rangle^{\a-1-|\nu|} 2^{-j|\g+ \b-\nu|}\left|(\partial^{\g+ \b-\nu}\chi)(2^{-j}\xi) \right|\\
   & \lesssim \sum_{\nu \leq \b} {{\b}\choose{\nu}}2^{(\a-1-|\nu|) j} 2^{-j} \lesssim 2^{j(\a -2)},
\end{align*}
because now $|\g+ \b-\nu|\geq 1$. Thus
\[
    \l_j^2 \|\partial^{\g}\chi_j \partial^\beta \mu\|_{W(\CF L^1, L^{\infty})}
    \lesssim \l_j^2 2^{j(\a -2)}=1.
\]
Finally, let's us prove \eqref{eq:partial_mu}, using the hypothesis $\langle \xi \rangle^{2-\a} \partial^{\g}\mu(\xi) \in W(\CF L^1, L^{\infty})(\R^d)$, $|\g|=2$, we have
\begin{align*}
    \l_j^2 \| \chi_j \partial^{\g}\mu\|_{W(\CF L^1, L^{\infty})}
    &\lesssim \l_j^2 \|\langle \xi \rangle^{2-\a} \partial^{\g}\mu\|_{W(\CF L^1, L^{\infty})} \|\langle \xi \rangle^{\a-2} \chi_j \|_{W(\CF L^1, L^{\infty})}\\
    &\lesssim \l_j^2 \sum_{\b \leq d+1}\|\partial^{\b}[\langle \xi \rangle^{\a-2} \chi_j]\|_{L^{\infty}}.
\end{align*}
Moreover, arguing as above
\begin{align*}
   | \partial^{\b}[\langle \xi \rangle^{\a-2}\chi_j(\xi)]|&
   \lesssim \sum_{\nu \leq \b} {{\b}\choose{\nu}} \langle \xi \rangle^{\a-2-|\nu|} 2^{-j| \b-\nu|}\left|(\partial^{\b-\nu}\chi)(2^{-j}\xi) \right|\\
   & \lesssim \sum_{\nu \leq \b} {{\b}\choose{\nu}}2^{(\a-2-|\nu|) j} 2^{0} \lesssim 2^{j(\a -2)}.
\end{align*}
 Thus
\[
    \l_j^2 \|\chi_j \partial^{\g} \mu\|_{W(\CF L^1, L^{\infty})}\lesssim \l_j^2 2^{j(\a -2)} =1.
\]
Hence, by \eqref{add0}, \eqref{add1} and \eqref{add2} we have
\begin{equation}\label{eq:Tj_bound}
    \|\tilde{T}^{(j)}f\|_{M^{p,q}}= \|A_j B_jf\|_{M^{p,q}}\lesssim
2^{-\frac{d(\alpha-2)}{2}j}\|f\|_{M^{p,q}},
\end{equation} for every $1\leq p,q \leq \infty$.\\

We now combine this estimate with those for the dilation
operator, given in Theorem
\ref{dilprop}. For
$p=1,\infty$ and $q=1$ they
read
\[
\|U_{\lambda_j}f\|_{M^{1,1}}\lesssim
2^{\frac{d(\alpha-2)}{2}j}
\|f\|_{M^{1,1}},
\]
\[
\|U_{\lambda_j}f\|_{M^{\infty,1}}\lesssim
\|f\|_{M^{\infty,1}},
\]
and
\[
\|U_{\lambda_j^{-1}}f\|_{M^{1,1}}\lesssim
\|f\|_{M^{1,1}},
\]
\[
\|U_{\lambda_j^{-1}}f\|_{M^{\infty,1}}\lesssim
2^{\frac{d(\alpha-2)}{2}j}
\|f\|_{M^{\infty,1}}.
\]
Therefore we obtain, for
$p=1,\infty$,
$$\|T^{(j)}f\|_{M^{p,1}}
\lesssim
2^{-\frac{d(\alpha-2)}{2}j}2^{\frac{d(\alpha-2)}{2}j}
\|f\|_{M^{p,1}}=\|f\|_{M^{p,1}}.$$

Finally, to sum over $j\geq1$ these
last estimates we take advantage of the
fact we are working with functions
which are localized in shells of the
frequency domain. Precisely, let $\chi$ as before, namely a smooth function satisfying
$\chi(\xi)=1$ for $1/2\leq |\xi|\leq2$
and $\chi(\xi)=0$ for $|\xi|\leq1/4$
and $|\xi|\geq4$ (so that
$\chi\psi=\psi$). With
$\chi_j(\xi)=\chi(2^{-j}\xi)$ and $p=1,\infty$ we have

$$\|T^{(j)}f\|_{M^{p,1}}=\|T^{(j)}(\chi(2^{-j}D)f)\|_{M^{p,1}}\lesssim\|\chi(2^{-j}D)f\|_{M^{p,1}},$$
so that Lemma \ref{sum_1}
gives us
\[
\left\|\sum_{j\geq1}
T^{(j)}f \right\|_{M^{p,1}}\leq
\sum_{j\geq1}\|T^{(j)}f\|_{M^{p,1}}\lesssim
\|f\|_{M^{p,1}}.
\]
\medskip
\textbf{Boundedness on
$M^{1,\infty}$ and
$M^{\infty,\infty}$
 for $\delta=\frac{d(\alpha-2)}{2}$.} We first establish the
following lemma (cf.\ 
\cite{cordero-nicola-rodino,sugimoto-tomita}).
\begin{lemma}\label{sum_infinito}
For $k\geq0$, let
$f_k\in \mathcal{S}(\R^d)$ satisfy
${\rm supp}\,\hat{f}_0\subset
B_2(0)$ and
\[
{\rm supp}\,\hat{f}_k\subset\{\xi\in\R^d:\
2^{k-1}\leq|\xi|\leq
2^{k+1}\},\quad k\geq1.
\]
Then, if the sequence $f_k$
is bounded in
$M^{p,\infty}(\R^d)$ for some
$1\leq p\leq\infty$, the
series $\sum_{k=0}^\infty
{f}_k$ converges in
${M}^{p,\infty}(\R^d)$ and
\begin{equation}\label{b0}
\left\|\sum_{k=0}^\infty
f_k \right\|_{M^{p,\infty}}\lesssim\sup_{k\geq0}\|f_k\|_{M^{p,\infty}}.
\end{equation}
 \end{lemma}
 \begin{proof}
The convergence of the series
$\sum_{k=0}^\infty {f}_k$ in
${M}^{p,\infty}(\R^d)$ is straightforward.
We now prove the desired
estimate.\par
 Choose a window
function $g$ with ${\rm
supp}\,\hat{g}\subset
B_{1/2}(0)$. We can write
\[
V_g(f_k)(x,\xi)=(\hat{f}_k\ast
M_{-x}\hat{\overline{g}})(\xi).
\]
Hence, ${\rm supp}\,V_g(f_0)\subset
B_{5/2}(0)\subset
B_{2^2}(0)$, and
\begin{align*}
{\rm supp}\,V_g({f}_k)&\subset\{(x,\xi)\in\R^{2d}:\
2^{k-1}-2^{-1}\leq|\xi|\leq
2^{k+1}+2^{-1}\}\\
&\subset\{(x,\xi)\in\R^{2d}:\
2^{k-2}\leq|\xi|\leq
2^{k+2}\},
\end{align*}
for $k\geq1$. Hence, for each
$\xi$, there are at most four
nonzero terms in the sum
$\sum_{k=0}^\infty
\|V_g(f_k)(\cdot,\xi)\|_{L^p}$.
Using this fact we obtain
\begin{align}
\|\sum_{k=0}^\infty
{f}_k\|_{M^{p,\infty}}&\asymp\|\sum_{k=0}^\infty
V_g({f}_k)\|_{L^{p,\infty}}\leq
\sup_{\xi\in\R^d}\sum_{k=0}^\infty
\|V_g(f_k)(\cdot,\xi)\|_{L^p}\nonumber\\
&\leq
4\sup_{k\geq0}\sup_{\xi\in\R^d}
\|V_g(f_k)(\cdot,\xi)\|_{L^p}=4\sup_{k\geq0}
\|V_g(f_k)\|_{L^{p,\infty}}\nonumber\\
&\asymp\sup_{k\geq0}
\|f_k\|_{M^{p,\infty}}.\nonumber
\end{align}
\end{proof}

We now consider the same
decomposition as above,
namely \eqref{decom1}, and
the operators
$\tilde{T}^{(j)}$ in
\eqref{decom2}, $j\geq1$. From \eqref{eq:Tj_bound} for $q= \infty$ we have the following estimate:
$$\|\tilde{T}^{(j)}f\|_{M^{p,\infty}}\leq
2^{-\frac{d(\alpha-2)}{2}j}\|f\|_{M^{p,\infty}}.$$
We then combine this estimate
with those for the dilation
operator which here read
\[
\|U_{\lambda_j}f\|_{M^{1,\infty}}\lesssim
2^{d(\alpha-2)j}
\|f\|_{M^{1,\infty}},
\]
\[
\|U_{\lambda_j}f\|_{M^{\infty,\infty}}\lesssim
2^{\frac{d(\alpha-2)}{2}j}\|f\|_{M^{\infty,\infty}},
\]
and
\[
\|U_{\lambda_j^{-1}}f\|_{M^{1,\infty}}\lesssim
2^{-\frac{d(\alpha-2)}{2}j}
\|f\|_{M^{1,\infty}},
\]
\[
\|U_{\lambda_j^{-1}}f\|_{M^{\infty,\infty}}\lesssim
\|f\|_{M^{\infty,\infty}}.
\]
Therefore we obtain, for
$p=1,\infty$,
$$\|T^{(j)}f\|_{M^{p,\infty}}
\lesssim
2^{-\frac{d(\alpha-2)}{2}j}2^{\frac{d(\alpha-2)}{2}j}
\|f\|_{M^{p,\infty}}=\|f\|_{M^{p,\infty}}.$$
We finally conclude by
applying Lemma
\ref{sum_infinito}: for
$p=1,\infty$,
\[
\left\|\sum_{j=1}^\infty
T^{(j)}f \right\|_{M^{p,\infty}}
\lesssim
\sup_{j\geq1}\|T^{(j)}f\|_{M^{p,\infty}}
\lesssim\|f\|_{M^{p,\infty}}.
\]
\textbf{Boundedness on
$M^{2,1}$ and
$\cM^{2,\infty}$ for
$\delta=0$}. Indeed, we will
prove boundedness on
$\cM^{2,q}$ for every $1\leq
q\leq\infty$ and $\delta=0$.
This is a special case of the
following result.
\begin{proposition}
Any Fourier multiplier $T$ with
symbol $\sigma\in L^\infty$
is bounded on $\cM^{2,q}$ for
every $1\leq q\leq\infty$.
\end{proposition}
\begin{proof}
This follows by a direct
computation. Namely
\begin{align*}
\|Tf\|_{M^{2,q}}&=\|\|M_x
\hat{\overline{g}}\ast
 \left(\sigma\hat{f}\right)\|_{L^2_x}\|_{L^q}\\
 &=\|\|\int e^{2\pi
ix(\xi-y)}\hat{\overline{g}}(\xi-y)\sigma(y)
\hat{f}(y)dy\|_{L^2_x}\|_{L^q_{\xi}}\\
   &=\|\|\sigma
\hat{f}T_\xi\hat{\overline{g}}\|_{L^2}\|_{L^q_{\xi}},
\end{align*}
where we used Parseval's
formula. In particular, this
computation with
$\sigma\equiv1$ gives
$\|f\|_{M^{2,q}}=\|\|
\hat{f}T_\xi\hat{g}\|_{L^2}\|_{L^q_{\xi}}$,
so we deduce at once the
desired estimate
$$\|Tf\|_{M^{2,q}}\lesssim
\|\sigma\|_{L^\infty}\|f\|_{M^{2,q}}.$$
\end{proof}
\section*{Acknowledgments}
We would like to thank Elena Cordero for inspiring discussions on the subject of this paper. 
The second author was partially supported by the  Generalitat Valenciana (Project VALi+d Pre Orden 64/2014 and Orden 86/2016)



\end{document}